\documentclass[12pt]{amsart}

\usepackage{amsmath,amssymb,amsthm,mathtools}
\usepackage{lmodern}
\usepackage{microtype}
\usepackage[hidelinks]{hyperref}
\usepackage[nameinlink,capitalise,noabbrev]{cleveref}

\newtheorem{theorem}{Theorem}[section]
\newtheorem{lemma}[theorem]{Lemma}

\theoremstyle{definition}

\newtheorem{remark}[theorem]{Remark}
\numberwithin{equation}{section}

\newcommand{\F}{\mathbb F}

\newcommand{\one}{\mathbf 1}

\newcommand{\RN}{\operatorname{RN}}
\newcommand{\Cr}{\mathrm{Cr}}
\newcommand{\dd}{\,\mathrm d}
\newcommand{\ee}{\mathrm e}
\newcommand{\abs}[1]{\lvert#1\rvert}
\newcommand{\norm}[1]{\lVert#1\rVert}
\DeclareMathOperator{\codeg}{codeg}
\DeclareMathOperator{\Rad}{Rad}

\title[Hamiltonicity in arithmetic graphs]{Hamiltonicity in graphs defined by primes and primitive elements}

\author{Yue-Feng She}
\address{Department of Applied Mathematics, Nanjing Forestry University,
Nanjing 210093, People's Republic of China}
\email{she.math@smail.nju.edu.cn}


\begin{document}

\begin{abstract}
A prime circle of order $2n$ is a circular ordering of $1,\ldots,2n$ such
that the sum of every two adjacent terms is prime. We prove that a prime
circle exists for every sufficiently large $n$. The proof is based on a
perfect matching and robust expansion. We also study Hamilton cycles in graphs and
digraphs defined by primitive sums and differences over finite fields. In
particular, the primitive-sum graph on $\F_q$ is Hamiltonian for every prime
power $q>18\,888\,871$, and for the graph on a full prime field $\F_p$, the bound improves to
$p>61$.
\end{abstract}

\maketitle

\section{Introduction}\label{sec:introduction}

Questions concerning permutations satisfying arithmetic restrictions have
a long history in combinatorics and number theory. For a positive integer $m$, write
\[
[m]=\{1,2,\ldots,m\}.
\]
A prime circle of order $2n$ is a circular ordering of $[2n]$ such that the
sum of every two adjacent terms is prime. In 1982, Filz \cite{Filz} posed the prime circle conjecture, which asserts that a prime circle of order $2n$ exists for every positive integer $n$. The conjecture is also recorded in Guy's book \cite{Guy}.

As usual, we denote by $V(G)$ and $E(G)$ the vertex set and the edge set of a graph $G$ respectively. For each positive integer $n$, let
\[
  O_n=\{1,3,\ldots,2n-1\},\qquad
  E_n=\{2,4,\ldots,2n\},
\]
and define the prime-sum graph $\Gamma_n$ on $[2n]$ by
\[
  xy\in E(\Gamma_n)
  \quad\Longleftrightarrow\quad
  x\neq y \text{ and } x+y\text{ is prime}.
\]
Apart from the sum $1+1=2$, parity makes $\Gamma_n$ bipartite with parts $O_n$ and $E_n$. For $n\geq2$, a prime circle of order $2n$ is
precisely a Hamilton cycle in $\Gamma_n$.

Greenfield and Greenfield \cite{GreenfieldGreenfield} proved that the prime-sum graph on $[m]$ has a
perfect matching if and only if $m$ is even.
Chen, Fu and Guo \cite{ChenFuGuo} proved that a prime circle of order $2n$
exists if there are primes $p_1<p_2\leq2n$ such that $2n+p_1$ and
$2n+p_2$ are prime and
\[
  \gcd\!\left(\frac{p_2-p_1}{2},n\right)=1.
\]
Together with bounded gaps between primes, this gives prime circles for
infinitely many orders.  More recently, Shang, Li and Zhang \cite{ShangLiZhang} constructed a
two-sided ordering of all positive integers in which the sum of every two
adjacent terms is prime. The question whether a prime
circle of order $2n$ exists for every positive integer $n$ remains open. Our first result proves the
conjecture for all sufficiently large $n$.

\begin{theorem}\label{thm:prime-circle}
There is an integer $n_0$ such that $\Gamma_n$ is Hamiltonian for every
$n\geq n_0$. Equivalently, a prime circle of order $2n$ exists for every
sufficiently large $n$.
\end{theorem}

Let $q$ be a prime power and $\F_q$ be the finite field of $q$ elements. Write $\mathcal P_q$ for the set of primitive
elements of $\F_q$, that is, the generators of the multiplicative group $\F_q^*=\F_q\setminus\{0\}$. Define the graph
$H_q$ on $\F_q$ by
\[
  xy\in E(H_q)
  \quad\Longleftrightarrow\quad
  x\ne y\ \text{ and }\ x+y\in\mathcal P_q.
\]
For an odd prime $p=2m+1$, define a graph $H_p^+$ and a digraph $H_p^-$,
both on $[m]$, by
\[
  xy\in E(H_p^+)
  \quad\Longleftrightarrow\quad
  x\ne y\ \text{ and }\ x+y\in\mathcal P_p,
\]
and
\[
  xy\in E(H_p^-)
  \quad\Longleftrightarrow\quad
  x-y\in\mathcal P_p.
\]
Sun and Hou \cite{Sun} conjectured that these three graphs are Hamiltonian,
apart from a finite list of small fields. Our second result is as follows.

\begin{theorem}\label{thm:primitive-main}
The following statements hold.
\begin{enumerate}
  \item If $q>18\,888\,871$ is a prime power, then $H_q$ is Hamiltonian.
  \item If $p>562\,582\,021$ is prime, then $H_p^+$ is Hamiltonian.
  \item If $p>3.661\times10^{52}$ is prime, then $H_p^-$ contains a
  directed Hamilton cycle.
\end{enumerate}
\end{theorem}

\begin{remark}
For $H_p$, a direct construction gives a stronger result. In fact, by Cohen, Oliveira e Silva and Trudgian
\cite[Theorem~1]{CohenOliveiraTrudgian}, every prime $p>61$ has an element
$x\in\F_p$ for which $x,x+1,x+2$ are primitive. Put $m=(p-1)/2$ and
define
\[
  a_{2k+1}=\frac x2+2k+1\quad(0\leq k\leq m),
\]
and
\[
  a_{2k}=\frac x2-2k+1\quad(1\leq k\leq m).
\]
One may easily verify that this ordering is a Hamilton cycle of $H_p$.
\end{remark}

\section{Graph-theoretic and analytic tools for the prime-circle problem}
\label{sec:graph-inputs}

For a digraph $D$ on $N$ vertices, a set
$S\subseteq V(D)$ and a real number $\nu>0$, define the robust outneighborhood
\[
  \RN^+_{\nu,D}(S)
  =\{v\in V(D):\abs{N_D^-(v)\cap S}\geq\nu N\},
\]
where $N_D^-(v)$ denotes the inneighborhood of $v$. The digraph is a \emph{robust $(\nu,\tau)$-outexpander}, where $0<\nu\leq\tau<1$ if
\[
  \abs{\RN^+_{\nu,D}(S)}>\abs S+\nu N
\]
whenever $\tau N\leq\abs S\leq(1-\tau)N$. The minimum semidegree of $D$ is
\[
  \delta^0(D)=\min\{\delta^+(D),\delta^-(D)\},
\]
where $\delta^+(D)$ and $\delta^-(D)$ are respectively the minimum outdegree and minimum indegree of $D$.

We use the following theorem of Lo and Patel
\cite[Theorem~3]{LoPatel}.

\begin{theorem}\label{thm:lo-patel}
Let $N$ be a positive integer and let $\nu,\tau,\eta\in(0,1)$ satisfy
\[
 4\left(\frac{\log^2N}{N}\right)^{1/13}<\nu\leq\tau\leq\frac{\eta}{16}<\frac1{16}.
\]
Let $D$ be an $N$-vertex digraph with $\delta^0(D)>\eta N$ which is a robust $(\nu,\tau)$-outexpander. Then, for every $\nu N/2\leq\ell\leq N$ and every vertex $v$, the digraph $D$ contains a directed cycle of length $\ell$ through $v$.
\end{theorem}

We next record the three number-theoretic lemmas used for prime circles.
The first gives a uniform estimate for the vertex degrees of $\Gamma_n$.

\begin{lemma}\label{lem:prime-degrees}
Uniformly for $v\in V(\Gamma_n)$,
\[
  d_{\Gamma_n}(v)=(2+o(1))\frac n{\log n}.
\]
\end{lemma}

\begin{proof}
As usual, $\pi(x)$ counts the number of primes not exceeding $x$. The prime number theorem gives
$$
\pi(x)=\frac{x}{\log x}+O\left(\frac{x}{\log^2 x}\right).
$$
For each $v\in V(\Gamma_n)$, the neighbors of $v$ correspond to the primes in $(v,v+2n]\setminus\{2v\}$. Thus
\[
  d_{\Gamma_n}(v)=\pi(v+2n)-\pi(v)+O(1).
\]
Put $A=n/(\log n)^2$.  If $v<A$, then
\[
  \pi(v)\leq v=o\left(\frac{n}{\log n}\right),
\]
while
\[
  \pi(v+2n)=(1+o(1))\frac{v+2n}{\log (v+2n)}=(2+o(1))\frac{n}{\log n}.
\]
If $v\geq A$, note that
\[
  \log t=(1+o(1))\log n
  \qquad (A\leq t\leq4n).
\]
Hence we obtain
\[
  \pi(v+2n)-\pi(v)=(1+o(1))\frac{v+2n}{\log n}-(1+o(1))\frac{v}{\log n}
  =(2+o(1))\frac{n}{\log n}.
\]
Thus the estimate holds uniformly for $1\leq v\leq2n$. 
\end{proof}

We use the upper bound sieve for prime tuples; see,
for example, \cite[Theorem 5.7]{HalberstamRichert}.

\begin{theorem}\label{thm:prime-tuple}
Let $g$ be a positive integer, and let $a_i,b_i(i=1,\ldots,g)$ be integers satisfying
$$
E:=\prod_{i=1}^ga_i\prod_{1\leq r<s\leq g}(a_rb_s-a_sb_r)\neq 0.
$$
Let $\rho(p)$ denote the number of solutions of
$$
\prod_{i=1}^g(a_ik+b_i)\equiv 0\pmod{p},
$$
and suppose that
$$
\rho(p)<p \text{ for all } p. 
$$
Then the number of integers $k$ in an interval of length $y$ for which all the $a_ik+b_i$ are prime is at most
\begin{align*}
&2^gg!\prod_{p}\left(1-\frac{\rho(p)}{p}\right)\left(1-\frac{1}{p}\right)^{-g}\frac{y}{\log^g y}\\
&\times\left(1+O_g\left(\frac{\log\log 3y+\log\log 3|E|}{\log y}\right)\right).
\end{align*}
\end{theorem}

\begin{lemma}\label{lem:prime-codegree}
If $x\ne x'$ lie in the same part of $\Gamma_n$, then
\[
  \codeg_{\Gamma_n}(x,x')
  \ll \frac n{\log^2n}
  \prod_{\substack{p\mid x-x'\\ p>2}}
  \frac{p-1}{p-2}
  \ll \frac{n(\log\log n)^2}{\log^2n}.
\]
\end{lemma}

\begin{proof}
Take $g=2$, $1\leq k\leq n$, $a_1=a_2=2$, $b_1=x+\delta$ and $b_2=x'+\delta$, where $\delta=0$ or $-1$ according as $x\in O_n$ or $x\in E_n$. Then 
$$
E=8(x'-x)\neq 0.
$$
Moreover, 
\begin{equation*}
\rho(2)=0, \qquad \rho(p)=\begin{cases} 1,\quad p\mid x'-x,\\2,\quad p\nmid x'-x\end{cases} ~ (p>2).
\end{equation*}
Consequently, Theorem \ref{thm:prime-tuple} gives the
first estimate. Since $\abs{x-x'}\leq2n$, the second follows from
\[
 \prod_{\substack{p\mid x-x'\\ p>2}}
 \frac{p-1}{p-2}
 \leq
 \left(\frac{\abs{x-x'}}{\varphi(\abs{x-x'})}\right)^2
 \ll(\log\log n)^2,
\]
where $\varphi$ denotes the Euler totient function.
\end{proof}

The von Mangoldt function $\Lambda(m)$ is defined for positive integers $m$ by
\begin{equation*}
\Lambda(m)=\begin{cases} \log p, &\mbox{if $m=p^k$ for some prime $p$ and integer $k\geq1$,}\\ 0,&\mbox{otherwise.}\end{cases}
\end{equation*}
For $w\geq2$, let
\begin{equation*}
  P(w)=\prod_{p<w}p,
  \qquad
  \Lambda_{\Cr,w}(m)
  =\begin{cases} \frac{P(w)}{\varphi(P(w))} &\text{if $(m,P(w))=1$,}\\ 0 &\text{otherwise.}\end{cases}
\end{equation*}
For a function $f\colon\mathbb Z\to\mathbb C$ and a positive integer $N$, the local Gowers uniformity norm $\norm{f}_{U^2[N]}$ is defined by
\[
  \left(
  \mathbb E_{\substack{n,h_1,h_2\in\mathbb Z\\
    n,n+h_1,n+h_2,n+h_1+h_2\in[N]}}
  f(n)\overline{f(n+h_1)}
  \overline{f(n+h_2)}f(n+h_1+h_2)
  \right)^{1/4},
\]
where $\mathbb E$ denotes the average over all triples satisfying the
displayed conditions.

The following is the von Mangoldt-function part of the uniformity
estimate of Tao and Ter\"av\"ainen
\cite[Theorem~1.3(ii)]{TaoTeravainen}.

\begin{theorem}\label{thm:prime-uniformity}
There is an absolute constant $c_0>0$ such that, for all sufficiently
large $N$ and every
\[
  2\leq w\leq\exp\bigl((\log N)^{1/10}\bigr),
\]
we have
\[
  \norm{\Lambda-\Lambda_{\Cr,w}}_{U^2[N]}
  \ll_{\mathrm{ineff}}
  (\log N)^{-c_0}+w^{-c_0}.
\]
The subscript $\mathrm{ineff}$ means the implied constants are permitted to be ineffective.
\end{theorem}

\begin{lemma}\label{lem:prime-bilinear}
There is an absolute constant $c>0$ such that the following holds for
all sufficiently large $N$. Let $X,Y\subseteq[N]$ and suppose that
$x+y\leq N$ for every $(x,y)\in X\times Y$. With
\[
  w=(\log N)^{1/100},
\]
we have
\[
  \sum_{x\in X}\sum_{y\in Y}\Lambda(x+y)
  =
  \sum_{x\in X}\sum_{y\in Y}\Lambda_{\Cr,w}(x+y)
  +O_{\mathrm{ineff}}\!\left(
    N(\log N)^{-c}\sqrt{\abs X\abs Y}
  \right).
\]
\end{lemma}

\begin{proof}
Put
\[
  f=(\Lambda-\Lambda_{\Cr,w})\one_{[N]}.
\]
We write $\ee(t)=e^{2\pi i t}$ and define
\[
  \widehat f(\theta)
  =\sum_{m\in\mathbb Z}f(m)\ee(m\theta).
\]
Cauchy--Schwarz's inequality gives
\begin{align*}
  \abs{\widehat f(\theta)}^4
  &=\abs{\sum_{-N+1\leq h\leq N-1} e(h\theta)\sum_{m\in\mathbb{Z}}f(m+h)\overline{f(m)}}^2 \\
  &\leq
  (2N-1)
  \sum_{h\in\mathbb Z}
  \abs{\sum_{m\in\mathbb Z}
    f(m+h)\overline{f(m)}}^2  \\
  &\ll
  N^4\norm{\Lambda-\Lambda_{\Cr,w}}_{U^2[N]}^4.
\end{align*}
Consequently, by \cref{thm:prime-uniformity},
\[
  \sup \abs{\widehat f(\theta)}
  \ll N\norm{\Lambda-\Lambda_{\Cr,w}}_{U^2[N]}\ll_{\mathrm{ineff}}
  N\bigl((\log N)^{-c_0}+w^{-c_0}\bigr)
\]
for some constant $c_0$. For $w=(\log N)^{1/100}$, this becomes
\[
  \sup \abs{\widehat f(\theta)}
  \ll_{\mathrm{ineff}}N(\log N)^{-c},
  \qquad c=\frac{c_0}{100}.
\]
Furthermore, H\"older's inequality gives
\begin{align*}
  &\quad\left|\sum_{x\in X}\sum_{y\in Y}f(x+y)\right|\\
  &=
  \left|\int_0^1\widehat f(\theta)
  \left(\sum_{x\in X}\ee(-x\theta)\right)
  \left(\sum_{y\in Y}\ee(-y\theta)\right)\dd\theta\right|\\
  &\leq
  \sup \abs{\hat{f}(\theta)} \left(\int_{0}^1\abs{\sum_{x\in X}\ee(-x\theta)}^2d\theta\right)^{1/2} \left(\int_{0}^1\abs{\sum_{y\in Y}\ee(-y\theta)}^2d\theta\right)^{1/2}\\
  &\ll_{\mathrm{ineff}}
  N(\log N)^{-c}\sqrt{\abs X\abs Y}.
\end{align*}
The result follows from the definition of $f$.
\end{proof}

\section{Edge-count estimates for the prime-sum graph}
\label{sec:edge-count}

Recall that $N_G(v)$ denotes the neighborhood of a vertex $v$ in the graph $G$. 
For a graph or digraph $G$ and arbitrary sets $X,Y\subset V(G)$, define
\[
 e_G(X,Y)
 =\bigl|\{(x,y)\in X\times Y:xy\in E(G)\}\bigr|.
\] 

Choose a sufficiently small absolute constant $\rho_0>0$ and put
\[
  \rho_n=\frac{\rho_0}{(\log\log n)^2}.
\]

\begin{lemma}\label{lem:prime-small-concentration}
For all sufficiently large $n$, if $X$ lies in one
part of $\Gamma_n$ and $U$ lies in the opposite
part with $\abs X\leq\rho_n n$ and $\abs U\leq2\abs{X}$, then 
\[
  e_{\Gamma_n}(X,U)\leq\frac{n}{2\log n}\abs{X}.
\]
\end{lemma}

\begin{proof}
The assertion is trivial when $x=\emptyset$, so assume $X\neq \emptyset$. Given a vertex $y$, write 
\[
  d_X(y)=\abs{N_{\Gamma_n}(y)\cap X}.
\]
Let $\delta_{xy}$ be $1$ or $0$ according as $xy\in E(\Gamma_n)$ or not. Note that
\begin{align*}
  \sum_y d_X(y)^2=\sum_y\left(\sum_{x\in X} \delta_{xy}\right)^2
  =\sum_{x\in X}d_{\Gamma_n}(x)
   +\sum_{\substack{x,x'\in X\\x\ne x'}}
    \codeg_{\Gamma_n}(x,x').
\end{align*}
Hence by Lemma \ref{lem:prime-degrees} 
and Lemma \ref{lem:prime-codegree}, there is a constant $C$ such that
\[
  \sum_y d_X(y)^2
  \leq C\abs{X}\left(
    \frac{n}{\log n}
    +\frac{n(\log\log n)^2\abs{X}}{\log^2n}
  \right).
\]
By Cauchy--Schwarz's inequality, we have
\begin{align*}
  e_{\Gamma_n}(X,U)^2
  &\leq\abs U\sum_y d_X(y)^2\\
  &\leq2C\abs{X}^2\left(
    \frac{n}{\log n}
    +\frac{n(\log\log n)^2\abs{X}}{\log^2n}
  \right).
\end{align*}
After division by $(n\abs{X}/\log n)^2$, by the upper bound for $\abs{X}$, the right-hand side is at most
\[
  2C\left(\frac{\log n}{n}+\rho_0\right).
\]
Choose $\rho_0$ sufficiently small in terms of $C$ and then take $n$
sufficiently large.  This proves the assertion.  
\end{proof}

\begin{lemma}\label{lem:local-spectrum}
Let $Q$ be an odd square-free integer divisible by $3$, and let $\varepsilon_Q$ be the principal Dirichlet character modulo $Q$. Let $K_Q$ be
the $Q\times Q$ matrix with $(i,j)$-entry 
$
  \varepsilon_Q(i+j).
$
Then its largest singular value is $\varphi(Q)$, with constant singular vectors, and every other singular value is at most $\varphi(Q)/2$.
\end{lemma}

\begin{proof}
For $1\leq r\leq Q$, define the vectors 
$$
v_r=\frac{1}{\sqrt{Q}}\left(\ee\!\left(\frac{r}{Q}\right),\ee\!\left(\frac{2r}{Q}\right),\ldots,\ee\!\left(\frac{rQ}{Q}\right)\right)^{\mathsf T}.
$$
Note that for any $1\leq r,j\leq Q$
$$
\sum_{k=1}^Q\varepsilon_Q(j+k)\ee\!\left(\frac{rk}{Q}\right)=c_Q(r)\ee\!\left(\frac{-rj}{Q}\right),
$$
where $c_Q(r)$ are the Ramanujan sums
\[
  \sum_{\substack{u\bmod Q\\(u,Q)=1}}
  \ee\!\left(\frac{ru}{Q}\right).
\]
Hence $$K_Qv_r=c_Q(r)\overline{v_r}.$$
Since the matrix with column vectors $v_1,v_2,\ldots,v_Q$ is unitary, the singular values of $K_Q$ are the absolute values of $c_Q(r)$. Since $Q$ is square-free,
\[
  \abs{c_Q(r)}
  =\frac{\varphi(Q)}{\varphi(Q/(Q,r))}.
\]
For $r=Q$ this is $\varphi(Q)$.  If $r\ne Q$, then
$Q/(Q,r)>1$.  This divisor is odd, so its totient is at least $2$.
The assertion follows.
\end{proof}

\begin{lemma}\label{lem:local-density}
Let
\[
  w=(\log4n)^{1/100},
  \qquad
  Q=\frac{P(w)}{2}.
\]
Suppose $X\subseteq O_n$ and $Y\subseteq E_n$ with
\[
  \abs X,\abs Y\geq\rho_n n,
  \qquad
  \abs X+\abs Y\geq n-\frac n{\log^2n}.
\]
Then, for all sufficiently large $n$,
\[
  \sum_{\substack{a,b\bmod Q\\(a+b,Q)=1}}
  \abs{X\cap(a\bmod Q)}\abs{Y\cap(b\bmod Q)}
  \geq\frac{\varphi(Q)}{3Q}\abs X\abs Y.
\]
\end{lemma}

\begin{proof}
Let $A$ and $B$ be $Q$-dimensional vectors with entries
\[
  A_i=\abs{X\cap(i\bmod Q)},
  \quad
  B_i=\abs{Y\cap(i\bmod Q)},
  \quad 1\leq i\leq Q.
\]
It is trivial that
\[
  0\leq A_i,B_i\leq\frac nQ+1.
\]
Let $\one$ denote the all-ones vector. Write
\[
  A=\frac{\abs X}{Q}\mathbf1+A_0,
  \qquad
  B=\frac{\abs Y}{Q}\mathbf1+B_0.
\]
It is easy to see that $A_0,B_0$ are orthogonal to the constant vector.  By Lemma
\ref{lem:local-spectrum},
\begin{equation*}
  A^{\mathsf T}K_QB
  \geq\frac{\varphi(Q)}Q\abs X\abs Y
  -\frac{\varphi(Q)}2\norm{A_0}_2\norm{B_0}_2.
\end{equation*}
Write
\[
  \abs X=\alpha n,\qquad \abs Y=\beta n.
\]
The above bound for $A_i$ gives
\begin{align*}
  \norm{A_0}_2^2
  &=\sum_{i=1}^QA_i^2-\frac{\abs X^2}Q\\
  &\leq\left(\frac nQ+1\right)\abs X-\frac{\abs X^2}Q\\
  &=\frac{n^2}Q\alpha\left(1-\alpha+\frac Qn\right),
\end{align*}
and similarly for $B_0$.

Put $\delta=Q/n$. The elementary estimate $\log Q=O(w)$ gives
$\delta=n^{-1+o(1)}$. Also, for large $n$, $\alpha\beta\ge\max\{\alpha,\beta\}\rho_n\geq\rho_n/3$.
Finally,
\begin{align*}
 A^{\mathsf T}K_QB&\geq\frac{\varphi(Q)}Q\abs X\abs Y
  \left(1-\frac{1}{2}\sqrt{\frac{(1-\alpha+\delta)(1-\beta+\delta)}{\alpha\beta}}\right) \\
 &\geq\frac{\varphi(Q)}Q\abs X\abs Y
  \left(1-\frac{1}{2}\sqrt{1+(\alpha\beta)^{-1}\left(\frac{1}{\log^2n}+2\delta+\delta^2\right)}\right)\\
 &=\frac{\varphi(Q)}Q\abs X\abs Y
  \left(\frac12-o(1)\right),
\end{align*}
which is stronger than the stated bound.
\end{proof}

\begin{lemma}\label{prop:prime-large-mixing}
If $X\subseteq O_n$ and $Y\subseteq E_n$
satisfy
\[
  \abs X,\abs Y\geq\rho_n n,
  \qquad
  \abs X+\abs Y\geq n-\frac n{\log^2n},
\]
then there is an absolute constant $c_1>0$ such that for
all sufficiently large $n$, 
\[
  e_{\Gamma_n}(X,Y)\geq c_1\frac{\abs X\abs Y}{\log n}.
\]
\end{lemma}

\begin{proof}
Take $N=4n$ and $w=(\log N)^{1/100}$, and write $P(w)=2Q$ as in
Lemma \ref{lem:local-density}. Since every $x+y$ with
$(x,y)\in O_n\times E_n$ is odd,
\[
  (x+y,P(w))=1 \quad \Longleftrightarrow \quad (x+y,Q)=1.
\]
Thus Lemma \ref{lem:local-density} yields
\[
  \sum_{x\in X}\sum_{y\in Y}\Lambda_{\Cr,w}(x+y)
  \geq\frac23\abs X\abs Y.
\]
By Lemma \ref{lem:prime-bilinear}, the error in replacing $\Lambda_{\Cr,w}$ by
$\Lambda$ is
\[
  O_{\mathrm{ineff}}\!\left(
    n(\log n)^{-c}\sqrt{\abs X\abs Y}
  \right)=o(\abs X\abs Y).
\]
Indeed, the hypotheses give
$\abs X\abs Y\geq(\rho_n/3)n^2$.

It remains to remove prime powers. Put
\[
  \vartheta(m)=\begin{cases}
    \log m,&m\text{ is prime},\\
    0,&\text{otherwise}.
  \end{cases}
\]
If
\[
  r_{X,Y}(m)=\abs{\{(x,y)\in X\times Y:x+y=m\}},
\]
then $r_{X,Y}(m)\leq n$, and hence
\begin{align*}
 \sum_{x\in X}\sum_{y\in Y}
 \bigl(\Lambda-\vartheta\bigr)(x+y)
 &\leq n\sum_{\substack{p^k\leq4n\\k\geq2}}\log p\\
 &\ll n^{3/2}\log n=o(\abs X\abs Y).
\end{align*}
Therefore, for all sufficiently large $n$,
\[
  \sum_{x\in X}\sum_{y\in Y}\vartheta(x+y)
  \geq\frac13\abs X\abs Y.
\]
Since every relevant prime is at most $4n$, division by $\log4n$
proves the result.
\end{proof}

\section{Proof of Theorem \ref{thm:prime-circle}}
\label{sec:prime-proof}

The following matching theorem is due to Greenfield and Greenfield
\cite[Theorem 1]{GreenfieldGreenfield}.

\begin{theorem}\label{lem:prime-matching}
For every $n\geq1$, the graph $\Gamma_n$ has a perfect matching.
\end{theorem}

Fix a perfect matching
\[
  M=\{o_i e_i:1\leq i\leq n\},
  \qquad o_i\in O_n,\quad e_i\in E_n.
\]
Define a loopless digraph $D_M$ on $[n]$ by
\[
  ij\in E(D_M)
  \quad\Longleftrightarrow\quad
  i\ne j\ \text{ and }\ o_i+e_j\text{ is prime}.
\]

\begin{lemma}\label{lem:prime-lifting}
If $D_M$ has a directed Hamilton cycle, then $\Gamma_n$ has a Hamilton
cycle.
\end{lemma}

\begin{proof}
If
$i_1\to i_2\to\cdots\to i_n\to i_1$ is a directed Hamilton cycle in
$D_M$, then
\[
  e_{i_1},o_{i_1},e_{i_2},o_{i_2},\ldots,
  e_{i_n},o_{i_n},e_{i_1}
\]
is a Hamilton cycle in $\Gamma_n$.
\end{proof}

\begin{lemma}\label{lem:prime-semidegree}
For all sufficiently large $n$ and every choice of $M$,
\[
  \delta^0(D_M)\geq\frac n{\log n}.
\]
\end{lemma}

\begin{proof}
For each $i$,
\[
  d^+_{D_M}(i)=d_{\Gamma_n}(o_i)-1,
  \qquad
  d^-_{D_M}(i)=d_{\Gamma_n}(e_i)-1.
\]
The deleted edge is exactly $o_ie_i$. The result follows from Lemma
\ref{lem:prime-degrees}.
\end{proof}

Fix a sufficiently small absolute constant $\kappa>0$ and put
\begin{equation}\label{eq:prime-parameters}
  \eta=\frac1{2\log n},
  \qquad
  \tau=\frac1{32\log n},
  \qquad
  \nu=\frac\kappa{\log^2n}.
\end{equation}

\begin{lemma}
\label{prop:prime-robust}
For all sufficiently large $n$ and every perfect matching $M$, the digraph $D_M$ is a robust
$(\nu,\tau)$-outexpander.
\end{lemma}

\begin{proof}
Let $S\subseteq[n]$ satisfy
\[
  \tau n\leq s:=\abs S\leq(1-\tau)n.
\]
Suppose for a contradiction that, with
\[
  R=\RN^+_{\nu,D_M}(S),
  \qquad T=[n]\setminus R,
  \qquad t=\abs T,
\]
we have $\abs R\leq s+\nu n$.  Then
\[
  s+t\geq n-\nu n,
  \qquad
  t\geq(\tau-\nu)n\geq\frac{\tau n}{2}
\]
for large $n$.  The definition of $T$ gives
\[
  e_{D_M}(S,T)<\nu nt.
\]

Let
\[
  X=\{o_i:i\in S\}\subseteq O_n,
  \qquad
  Y=\{e_j:j\in T\}\subseteq E_n.
\]
Only matching edges with an index in $S\cap T$ are counted in
$e_{\Gamma_n}(X,Y)$ but not in $e_{D_M}(S,T)$.  Consequently,
\begin{equation}\label{eq:prime-robust-upper}
  e_{\Gamma_n}(X,Y)<\nu nt+n.
\end{equation}

We distinguish three cases.

\emph{Case 1: $s\leq\rho_n n$.}
The set $U=E_n\setminus Y$ satisfies
\[
  \abs U=\abs R\leq s+\nu n\leq2s
\]
for large $n$, since $s\geq\tau n$ and $\nu=o(\tau)$.  By Lemma \ref{lem:prime-degrees} and Lemma
\ref{lem:prime-small-concentration},
\[
  e_{\Gamma_n}(X,Y)
  \geq\frac12\frac{sn}{\log n}
  \geq\frac{n^2}{64\log^2n}.
\]
This contradicts \eqref{eq:prime-robust-upper} when, for example,
$\kappa<1/128$ and $n$ is large.

\emph{Case 2: $t\leq\rho_n n$.}
Now $U=O_n\setminus X$ has
\[
  \abs U=n-s\leq t+\nu n\leq2t.
\]
Applying Lemma \ref{lem:prime-degrees} and Lemma \ref{lem:prime-small-concentration} with $Y$ in place of $X$
gives
\[
  e_{\Gamma_n}(X,Y)\geq\frac12\frac{tn}{\log n}.
\]
This again contradicts \eqref{eq:prime-robust-upper}, because
\[
  \frac{\nu nt+n}{tn/\log n}
  \leq\frac\kappa{\log n}
  +O\!\left(\frac{\log^2n}{n}\right)=o(1).
\]

\emph{Case 3: $s,t>\rho_n n$.}
The hypotheses of Lemma \ref{prop:prime-large-mixing} hold, so
\[
  e_{\Gamma_n}(X,Y)\geq c_1\frac{st}{\log n}.
\]
This is incompatible with \eqref{eq:prime-robust-upper}, since
\[
  \frac{c_1st/\log n}{\nu nt}
  =\frac{c_1(s/n)\log n}{\kappa}
  \geq\frac{c_1\rho_0\log n}
  {\kappa(\log\log n)^2}\longrightarrow\infty,
\]
and the additive term $n$ is negligible.  Every case is contradictory,
which proves robust outexpansion.
\end{proof}

\begin{proof}[Proof of \cref{thm:prime-circle}]
Choose a perfect matching $M$ by \cref{lem:prime-matching} and form
$D_M$. By Lemma \ref{lem:prime-semidegree}, we have
\[
  \delta^0(D_M)>\frac n{2\log n}=\eta n.
\]
By Lemma \ref{prop:prime-robust}, $D_M$ is a
robust $(\nu,\tau)$-outexpander. The parameters in
\eqref{eq:prime-parameters} satisfy
\[
  \nu\leq\tau=\frac\eta{16}<\frac1{16}
\]
and
\[
  4\left(\frac{\log^2n}{n}\right)^{1/13}
  <\frac\kappa{\log^2n}=\nu
\]
for all sufficiently large $n$.  Hence \cref{thm:lo-patel} gives a
directed Hamilton cycle in $D_M$.  By Lemma \ref{lem:prime-lifting}, this cycle
lifts to a Hamilton cycle of $\Gamma_n$.
\end{proof}

\section{Preliminaries for Theorem \ref{thm:primitive-main}}
\label{sec:finite-tools}

For a positive integer $r$, let $\omega(r)$ be the number of distinct
prime divisors of $r$, and put
\[
  W(r)=2^{\omega(r)}.
\]
Thus $W(r)$ is the number of square-free divisors of $r$.  We write
$\mu$ for the M\"obius function. Also, we define the radical of $r$ by
$$
\operatorname{Rad}(r)=\prod_{\substack{\ell\mid r\\ \ell \text{ prime}}}\ell.
$$

Let $e\mid q-1$. An element $x\in\F_q$ is \emph{$e$-free} if $x\ne0$
and $x=y^d$ with $d\mid e$ implies $d=1$. In particular, $x$ is
primitive exactly when it is $(q-1)$-free. Let $\widehat{\F_q^*}$ and $\widehat{\F_q}$ denote the groups of multiplicative and additive characters of $\F_q$, respectively. We extend every multiplicative
character of $\F_q^*$ to $0$ by setting $\chi(0)=0$. The standard
characteristic-function is
\begin{equation}\label{eq:e-free-characteristic}
 \one_{e\text{-free}}(x)
 =\frac{\varphi(e)}e
  \sum_{d\mid e}\frac{\mu(d)}{\varphi(d)}
  \sum_{\substack{\chi\in\widehat{\F_q^*}\\\operatorname{ord}\chi=d}}
  \chi(x).
\end{equation}

\begin{lemma}\label{lem:finite-gauss}
Let $e\mid q-1$ and let $\psi$ be a nontrivial additive character of
$\F_q$.  Then
\[
 \left|\sum_{\substack{x~ e\text{-free}}}\psi(x)\right|
 \leq\frac{\varphi(e)}eW(e)\sqrt q.
\]
\end{lemma}

\begin{proof}
Substituting \eqref{eq:e-free-characteristic}, we obtain
\[
 \sum_{\substack{x\in\F_q\\x\ e\text{-free}}}\psi(x)=\frac{\varphi(e)}{e}\sum_{d\mid e}\frac{\mu(d)}{\varphi(d)}
 \sum_{\substack{\chi\in\widehat{\F_q^*}\\ \operatorname{ord}(\chi)=d}}
 \sum_{x\in\F_q}\chi(x)\psi(x).
\]
For the principal multiplicative character, the inner sum equals $-1$.
For every nonprincipal multiplicative character it is a Gauss sum of
absolute value $\sqrt q$.  There are $\varphi(d)$ characters of order
$d$, and only square-free divisors $d$ contribute. Taking absolute values gives the
bound.
\end{proof}

Put
\begin{equation*}\label{eq:Lambda-q}
  \Lambda_q
  =\frac{\varphi(q-1)}{q-1}W(q-1)\sqrt q.
\end{equation*}
For $X,Y\subseteq\F_q$, put
\begin{align*}
 E^+(X,Y)&=\abs{\{(x,y)\in X\times Y:x+y\in\mathcal P_q\}},\\
 E^-(X,Y)&=\abs{\{(x,y)\in X\times Y:x-y\in\mathcal P_q\}}.
\end{align*}

\begin{lemma}\label{lem:finite-mixing}
For all $X,Y\subseteq\F_q$,
\[
  \left|E^+(X,Y)-\frac{\varphi(q-1)}q\abs X\abs Y\right|
  \leq\Lambda_q\sqrt{\abs X\abs Y},
\]
and
\[
  \left|E^-(X,Y)-\frac{\varphi(q-1)}q\abs X\abs Y\right|
  \leq\Lambda_q\sqrt{\abs X\abs Y}.
\]
\end{lemma}

\begin{proof}
Recall that for a function $f:\F_q\to\mathbb{C}$, Fourier inversion gives
\[
 f(t)=\frac1q\sum_{\psi\in\widehat{\F_q}}\widehat f(\psi)\psi(t),
\]
where
\[
 \widehat f(\psi)=\sum_{u\in\F_q}f(u)\overline{\psi(u)}.
\]
Let $f=\mathbf{1}_{(q-1)\text{-free}}$. By the definition of $E^+(X,Y)$ and Fourier inversion, we have
\begin{align*}
 E^+(X,Y)
 &=\sum_{x\in X}\sum_{y\in Y}f(x+y)\\
 &=\frac1q\sum_{\psi\in\widehat{\F_q}}\widehat f(\psi)
 \left(\sum_{x\in X}\psi(x)\right)
 \left(\sum_{y\in Y}\psi(y)\right).
\end{align*}
The trivial additive character contributes
\[
 \frac{\varphi(q-1)|X||Y|}{q}.
\]
If $\psi$ is nontrivial, Lemma~\ref{lem:finite-gauss} gives $|\widehat f(\psi)|\leq\Lambda_q$. Cauchy's inequality then yields
\begin{align*}
 \left|E^+(X,Y)-\frac{\varphi(q-1)|X||Y|}{q}\right|
 &\leq\frac{\Lambda_q}{q}
 \sum_{\substack{\psi\in\widehat{\F_q}\\\psi\text{ nontrivial}}}
 \left|\sum_{x\in X}\psi(x)\right|
 \left|\sum_{y\in Y}\psi(y)\right|\\
 &\leq\frac{\Lambda_q}{q}
 \sqrt{(q|X|-|X|^2)(q|Y|-|Y|^2)}\\
 &\leq\Lambda_q\sqrt{|X||Y|}.
\end{align*}
For $E^-(X,Y)$ one uses $\psi(x-y)=\psi(x)\overline{\psi(y)}$, and the same argument applies.
\end{proof}

An interval in $\F_p$ is a set of consecutive residues.  Put
\begin{equation*}\label{eq:L-function}
  L(x)=\frac2{\pi^2}\log x+1.
\end{equation*}
The explicit P\'olya--Vinogradov inequality of Frolenkov and
Soundararajan \cite[Theorem~2]{FrolenkovSoundararajan} implies that, for
$p>1200$, every nonprincipal character $\chi$ modulo $p$ and every
interval $I\subseteq\F_p$ satisfy
\begin{equation}\label{eq:PV-common}
  \left|\sum_{a\in I}\chi(a)\right|\leq\sqrt p\,L(p).
\end{equation}

\begin{lemma}
\label{cor:primitive-interval}
Let $p>1200$ be prime, let $I\subseteq\F_p$ be an interval, and put
$I^*=I\setminus\{0\}$.  Then
\[
 \left|\abs{\mathcal P_p\cap I}
 -\frac{\varphi(p-1)}{p-1}\abs{I^*}\right|
 \leq\frac{\varphi(p-1)}{p-1}
 W(p-1)\sqrt p\,L(p).
\]
\end{lemma}

\begin{proof}
Sum \eqref{eq:e-free-characteristic} over $I$ with $e=p-1$. The
principal character contributes
$\varphi(p-1)\abs{I^*}/(p-1)$. Bound every nonprincipal contribution by
\eqref{eq:PV-common} and the stated estimate follows.
\end{proof}

For undirected graphs we use the following theorem of Chv\'atal and
Erd\H{o}s \cite[Theorem~1]{ChvatalErdos}.

\begin{theorem}
\label{thm:chvatal-erdos}
Let $G$ be a finite simple graph on at least three vertices.  If its
vertex connectivity $\kappa(G)$ is at least its independence number
$\alpha(G)$, then $G$ is Hamiltonian.
\end{theorem}

\section{Primitive-sum graphs on a finite field}
\label{sec:Hq}

We first prove the finite-field part of
\cref{thm:primitive-main}.  The argument is a direct comparison between
independence and connectivity.

\begin{proof}[Proof of \cref{thm:primitive-main}(1)]
Let $I$ be an
independent set of size $s$ in $H_q$.  Partition $I$ into disjoint sets
$X,Y$ with sizes $\lfloor s/2\rfloor$ and $\lceil s/2\rceil$.  Since
$E^+(X,Y)=0$, Lemma \ref{lem:finite-mixing} gives
\[
  \sqrt{\abs X\abs Y}\leq\frac{\Lambda_q q}{\varphi(q-1)}.
\]
It follows that
\begin{equation}\label{eq:Hq-alpha}
  \alpha(H_q)\leq\frac{2\Lambda_q q}{\varphi(q-1)}+1.
\end{equation}

Suppose first that
$H_q$ is disconnected. Let $A$ be a smallest component and put
$B=\F_q\setminus A$. Then $\abs B\geq\abs A$ and the no-edge estimate as above gives
\[
  \abs A\leq\frac{\Lambda_q q}{\varphi(q-1)}.
\]
This contradicts the fact that every vertex of $A$ has degree at least $\varphi(q-1)-1$ whenever
$$\varphi(q-1)-1>\frac{\Lambda_q q}{\varphi(q-1)}.$$ 
Thus the numerical condition below also ensures
that $H_q$ is connected.

Now let $S$ be a vertex cut, let $A$ be a smallest component of
$H_q-S$, and put $B=\F_q\setminus(A\cup S)$.  There are no edges from
$A$ to $B$, and $\abs B\geq\abs A$, so again
$\abs A\leq\Lambda_q q/\varphi(q-1)$.  For $a\in A$, at most $\abs A-1$ of its
neighbors lie in $A$. Hence
\begin{equation}\label{eq:Hq-kappa}
  \kappa(H_q)\geq (\varphi(q-1)-1)-(\abs{A}-1)\geq\varphi(q-1)-\frac{\Lambda_q q}{\varphi(q-1)}.
\end{equation}
By \cref{thm:chvatal-erdos}, equations
\eqref{eq:Hq-alpha}--\eqref{eq:Hq-kappa} prove Hamiltonicity whenever
\[
  \varphi(q-1)\geq\frac{3q}{q-1}W(q-1)\sqrt q+1.
\]
By Lemma \ref{prop:numerical-bounds}(1), this inequality holds for every prime
power
$
  q>18\,888\,871.
$
\end{proof}

\section{Primitive-sum graphs on a half interval}
\label{sec:Hp-plus}

\begin{proof}[Proof of \cref{thm:primitive-main}(2)]
Let $p>1200$ be prime and put $m=(p-1)/2$. The independent-set argument from \cref{sec:Hq}, applied inside $[m]$,
gives
\begin{equation}\label{eq:Hpplus-alpha}
  \alpha(H_p^+)\leq\frac{2\Lambda_p p}{\varphi(p-1)}+1.
\end{equation}

Fix $i\in[m]$.  As $j$ ranges over $[m]$, the residues $i+j$ form an
interval of length $m$ not containing $0$.  By
Lemma \ref{cor:primitive-interval}, and allowing for the forbidden loop
$j=i$,
\begin{equation*}\label{eq:Hpplus-degree}
 \delta(H_p^+)
 \geq\frac{\varphi(p-1)}{2}
 -\Lambda_pL(p)-1.
\end{equation*}
The same smallest-component argument used for $H_q$ shows, once the graph
is connected, that
\begin{equation}\label{eq:Hpplus-kappa}
 \kappa(H_p^+)\geq
 \delta(H_p^+)-\frac{\Lambda_p p}{\varphi(p-1)}+1.
\end{equation}
The sufficient inequality obtained by comparing
\eqref{eq:Hpplus-alpha} and \eqref{eq:Hpplus-kappa} is
\begin{equation}\label{eq:Hpplus-sufficient}
 \varphi(p-1)\geq
 \frac{6W(p-1)p^{3/2}+2p-2}
 {p-1-2W(p-1)\sqrt p\,L(p)}>0,
\end{equation}
which holds for every
$p>562\,582\,021$ by Lemma \ref{prop:numerical-bounds}(2). Indeed, this condition gives
$\delta(H_p^+)>\Lambda_p p/\varphi(p-1)$, so it also supplies the connectivity used
in \eqref{eq:Hpplus-kappa}. The
Chv\'atal--Erd\H{o}s theorem completes the proof.
\end{proof}

\section{Primitive-difference digraphs}
\label{sec:Hp-minus}

\begin{proof}[Proof of \cref{thm:primitive-main}(3)]
Let $p>3.661\times10^{52}$ be prime, set $m=(p-1)/2$, and put
\[
  \sigma=\frac{9999}{10000},
  \qquad
  \eta=\frac{\sigma\varphi(p-1)}p,
  \qquad
  \tau=\frac\eta{16},
  \qquad
  \nu=\frac{\eta^2}{16}.
\]
Then $0<\nu<\tau=\eta/16<1/16$.

Let $S\subseteq[m]$ satisfy
$\tau m\leq\abs S\leq(1-\tau)m$, and put
\[
  T=[m]\setminus\RN^+_{\nu,H_p^-}(S).
\]
Then 
$$
E^-(S,T)<\nu m\abs T.
$$
Suppose that
$\abs{\RN^+_{\nu,H_p^-}(S)}\leq\abs S+\nu m$. Then
\[
  \abs T\geq(\tau-\nu)m.
\]
On the other hand, Lemma \ref{lem:finite-mixing} gives
\begin{align*}
 E^-(S,T)
 \geq\abs T\left(
   \frac{\tau m\varphi(p-1)}p
   -\Lambda_p\sqrt{\frac{1-\tau}{\tau-\nu}}
 \right).
\end{align*}
The first inequality in Lemma \ref{prop:numerical-bounds}(3) is precisely the
condition that the parenthesis is larger than $\nu m$.  This is a
contradiction.  Thus $H_p^-$ is a robust $(\nu,\tau)$-outexpander.

For a fixed $i\in[m]$, each of the sets $i-[m]$ and $[m]-i$ is an
interval of length $m$ containing $0$. By Lemma \ref{cor:primitive-interval},
\begin{equation*}\label{eq:Hpminus-degree}
 \delta^0(H_p^-)
 \geq\frac{\varphi(p-1)}{p-1}
 \left(\frac{p-3}{2}-W(p-1)\sqrt p\,L(p)\right).
\end{equation*}
The second inequality in Lemma \ref{prop:numerical-bounds}(3) makes the
right-hand side greater than
\[
  \frac{\sigma\varphi(p-1)(p-1)}{2p}=\eta m.
\]
The third inequality there gives
\[
  \nu>4\left(\frac{\log^2m}{m}\right)^{1/13}.
\]
Every hypothesis of \cref{thm:lo-patel} is now satisfied, so $H_p^-$ has
a directed Hamilton cycle.
\end{proof}

\section{Numerical verification of the finite-field thresholds}
\label{sec:numerics}

We collect here the arithmetic calculations used above. 

\begin{lemma}\label{prop:numerical-bounds}
Put $\sigma=9999/10000$.
\begin{enumerate}
\item For every prime power $q>18\,888\,871$,
\begin{equation*}\label{eq:num-i}
 \varphi(q-1)\geq
 \frac{3q}{q-1}W(q-1)\sqrt q+1.
\end{equation*}

\item For every prime $p>562\,582\,021$,
\begin{equation*}\label{eq:num-ii-main}
 \varphi(p-1)\geq
 \frac{6W(p-1)p^{3/2}+2p-2}
 {p-1-2W(p-1)\sqrt p\,L(p)}>0.
\end{equation*}

\item For every prime $p>3.661\times10^{52}$, the following three
inequalities hold:
\begin{align}\label{eq:num-iii-mixing}
 \varphi(p-1)
 >\frac{32p^2W(p-1)\sqrt p}
 {\sigma(1-\sigma)(p-1)^2}
 \sqrt{\frac{p(16p-\sigma\varphi(p-1))}
 {\sigma\varphi(p-1)(p-\sigma\varphi(p-1))}},
\end{align}
\begin{align}\label{eq:num-iii-degree}
 \frac{p-3}{2}-\frac{\sigma(p-1)^2}{2p}
 >W(p-1)\sqrt p\,L(p),
\end{align}
\begin{align}\label{eq:num-iii-LP}
 \frac{\sigma^2\varphi(p-1)^2}{16p^2}
 >4\left(
 \frac{2\log^2((p-1)/2)}{p-1}
 \right)^{1/13}.
\end{align}
\end{enumerate}
\end{lemma}

\begin{proof}
Let $p_j$ be the $j$th prime and define
\[
 P_r=\prod_{j=1}^r p_j,
 \qquad
 \alpha_r=\frac{\varphi(P_r)}{P_r},
 \qquad
 \beta_r=\frac{\varphi(P_r)}{\sqrt{P_r}}.
\]
If $\omega(N)=r$, then
\begin{equation}\label{eq:primorial-bounds}
 N\geq P_r,
 \qquad
 \frac{\varphi(N)}N\geq\alpha_r,
 \qquad
 \frac{\varphi(N)}{\sqrt N}\geq\beta_r.
\end{equation}
Indeed, if $N=\prod_{j=1}^r\ell_j^{a_j}$ with
$\ell_1<\cdots<\ell_r$ primes and $a_1,\dots,a_r$ positive integers, then $\ell_j\geq p_j$ and
\[
 \frac{\varphi(N)}{\sqrt N}
 =\prod_{j=1}^r
 \ell_j^{(a_j-1)/2}\frac{\ell_j-1}{\sqrt{\ell_j}}.
\]
Each factor is minimized when $a_j=1$ and $\ell_j=p_j$.

\smallskip
\noindent\emph{Part (1).}
Put $N=q-1$ and $r=\omega(N)$. After division by $\sqrt N$, the assertion in part (1) becomes
\begin{equation}\label{eq:num-i-normalized}
 \frac{\varphi(N)}{\sqrt N}
 \geq3\cdot2^r\left(1+\frac1N\right)^{3/2}
 +\frac1{\sqrt N}.
\end{equation}
Here $N>18\,888\,870$.  If $r\leq7$, the right-hand side is less than
$385$, whereas
\[
 \frac{\varphi(N)}{\sqrt N}
 \geq\alpha_7\sqrt N
 >\frac{3072}{17017}\sqrt{18\,888\,870}>784.
\]

Suppose $r=8$.  If $N$ is square-free, a finite split according to the largest prime factor
shows that its increasing prime tuple is coordinatewise at least one of
\begin{gather*}
 (2,3,5,7,13,17,19,23),\\
 (2,3,5,7,11,13,23,29),\quad
 (2,3,5,7,11,17,19,29),\\
 (2,3,5,7,11,13,19,37),\quad
 (2,3,5,7,11,13,17,41).
\end{gather*}
The smallest resulting lower bound is
\[
\frac{\varphi(N)}{\sqrt N}
\geq
\prod_{\ell\in\{2,3,5,7,11,13,23,29\}}
\frac{\ell-1}{\sqrt\ell}
=
\frac{3\,548\,160}{\sqrt{20\,030\,010}}
>790.
\]
If $N$ is not square-free, then, apart from
$N=2P_8=19\,399\,380$, one has $N\geq23\,483\,460$ and hence
\[
 \frac{\varphi(N)}{\sqrt N}
 \geq\alpha_8\sqrt{23\,483\,460}>828.
\]
The exceptional value would give
\[
 q=19\,399\,381=67\cdot289\,543,
\]
which is not a prime power.  In all $r=8$ cases, the right-hand side of
\eqref{eq:num-i-normalized} is less than $769$.

Finally,
\[
 \beta_9=\frac{36\,495\,360}{\sqrt{223\,092\,870}}
 >4\cdot2^9.
\]
For $r\geq9$, each additional factor
$(p_j-1)/\sqrt{p_j}$ is greater than $2$, so
$\beta_r>4\cdot2^r$.  The right-hand side of
\eqref{eq:num-i-normalized} is smaller than $4\cdot2^r$.  This proves
part~(1).

\smallskip
\noindent\emph{Part (2).} Put
\[
  r=\omega(p-1),
  \qquad
  t=\frac{\varphi(p-1)}{p-1}.
\]
For $u>1$ and $a>0$, define
\[
  \mathcal H_r(u,a)=
  \frac{2^{r+1}\sqrt u\,L(u)}{u-1}
  +\frac{6\cdot2^r u^{3/2}}{a(u-1)^2}
  +\frac{2}{a(u-1)}.
\]
The assertion in part~(2) is equivalent to
\[
  \mathcal H_r(p,t)\leq1.
\]
Indeed, this inequality makes the first summand strictly smaller
than $1$, giving denominator positivity, and rearrangement yields
the required lower bound for $\varphi(p-1)$.

Direct differentiation shows that $\mathcal H_r(u,a)$ decreases
with $u>1$. Clearly, it also decreases with $a$ and increases with $r$.
We use the following numerical evaluations, where
\[
  \alpha_*=\alpha_8\frac{36}{37}.
\]
Direct substitution gives
\[
\begin{array}{c|r|c|c}
 r & u & a & \mathcal H_r(u,a)\\ \hline
 8 & 562\,582\,021   & \alpha_8    & <0.489\\
 9 & 562\,582\,021   & \alpha_*    & <0.998\\
 9 & 601\,380\,781   & \alpha_9    & <0.979\\
10 & 6\,469\,693\,231 & \alpha_{10} & <0.626
\end{array}
\]

If $r\leq8$, then $t\geq\alpha_8$, and hence
\[
  \mathcal H_r(p,t)
  \leq\mathcal H_8(562\,582\,021,\alpha_8)<1.
\]

Suppose next that $r=9$, and write the distinct prime divisors of
$p-1$ as $s_1<\cdots<s_9$. Unless
\[
  \operatorname{Rad}(p-1)=P_8q,
  \qquad q\in\{23,29,31\},
\]
we have $t\geq\alpha_*$. Indeed, if $s_9\geq37$, then
\[
  t\geq\alpha_8\frac{36}{37}=\alpha_*.
\]
Otherwise, outside the three exceptional patterns, necessarily
$s_8\geq23$ and $s_9\geq29$, so
\[
  t\geq
  \alpha_7\frac{22}{23}\frac{28}{29}
  >\alpha_*.
\]
Thus every nonexceptional pattern satisfies
\[
  \mathcal H_9(p,t)
  \leq\mathcal H_9(562\,582\,021,\alpha_*)<1.
\]
For the three exceptional patterns, observe that
\[
  p-1>562\,582\,020=2P_8\cdot29.
\]
Since $p-1$ is a multiple of its radical, we obtain, respectively,
\[
  p-1\geq3P_8\cdot23,
  \qquad
  p-1\geq3P_8\cdot29,
  \qquad
  p-1\geq2P_8\cdot31.
\]
Consequently,
\[
  p\geq2P_8\cdot31+1=601\,380\,781,
\]
and therefore
\[
  \mathcal H_9(p,t)
  \leq\mathcal H_9(601\,380\,781,\alpha_9)<1.
\]

Finally, suppose that $r\geq10$. Then
\[
  \mathcal H_r(p,t)
  \leq\mathcal H_r(P_r+1,\alpha_r).
\]
The expression on the right decreases with $r$. To verify this,
note that Euclid's argument gives
$p_{r+1}\leq P_r+1$, whence
\[
  L(P_{r+1}+1)\leq2L(P_r+1).
\]
The ratios of the three summands in
$\mathcal H_{r+1}(P_{r+1}+1,\alpha_{r+1})$
to their counterparts in $\mathcal H_r(P_r+1,\alpha_r)$
are at most
\[
  \frac4{\sqrt{p_{r+1}}},
  \qquad
  \frac{2\sqrt{p_{r+1}}}{p_{r+1}-1},
  \qquad
  \frac1{p_{r+1}-1},
\]
respectively. All three are smaller than $1$ since $p_{r+1}\geq 31$. Since
$P_{10}+1=6\,469\,693\,231$, it follows that
\[
  \mathcal H_r(p,t)
  \leq\mathcal H_{10}(P_{10}+1,\alpha_{10})
  <0.626<1.
\]
This proves part~(2).

\smallskip
\noindent\emph{Part (3).}
Put $X=3.661\times10^{52}$, $N=p-1$, and $r=\omega(N)$.
We first establish \eqref{eq:num-iii-LP}. Define
\[
 F_r(x)=
 \left(\frac{\sigma\alpha_r(x-1)}{8x}\right)^{26}
 \frac{x-1}{2\log^2((x-1)/2)}.
\]
By \eqref{eq:primorial-bounds}, it suffices to show that
$F_r(p)>1$. For $x>2e^2+1$, the function $F_r(x)$ increases
with $x$, while for fixed $x$ it decreases with $r$.
Direct evaluation gives
\[
 F_{32}(X)>1.00009,
 \qquad
 F_{33}(P_{33}+1)>1.6086.
\]
Thus, if $r\leq32$, then
\[
 F_r(p)\geq F_{32}(p)>F_{32}(X)>1.
\]

For $r\geq33$, we have $p\geq P_r+1$. We claim that
$F_r(P_r+1)$ increases with $r$ in this range.
Again we have $p_{r+1}\leq P_r+1$.
Consequently,
\[
 \log(P_{r+1}/2)<3\log(P_r/2).
\]
Therefore, by Bernoulli's inequality, we obtain
\[
 \frac{F_{r+1}(P_{r+1}+1)}{F_r(P_r+1)}
 >
 \frac{p_{r+1}}{9}\left(1-\frac{1}{p_{r+1}}\right)^{26}
 \geq\frac{p_{r+1}-26}{9}>1.
\]
It follows that
\[
 F_r(p)\geq F_r(P_r+1)
 \geq F_{33}(P_{33}+1)>1.
\]
This proves \eqref{eq:num-iii-LP}.

For the remaining inequalities, we use the elementary bound
\[
 W(N)<3N^{1/3}.
\]
Indeed,
\[
 \frac{W(N)^3}{N}
 \leq\prod_{\substack{\ell\mid N\\ \ell\ \mathrm{prime}}}
       \frac8\ell
 \leq\frac{8^4}{2\cdot3\cdot5\cdot7}<27,
\]
since each factor corresponding to a prime $\ell\geq11$
is less than $1$.

Put
\[
 t=\frac{\sigma\varphi(N)}p,
 \qquad
 \delta=1-\sigma=\frac1{10000}.
\]
As $N$ is even, we have $0<t<1/2$. The inequality
\eqref{eq:num-iii-LP}, already proved, gives
\[
 t>
 8\left(\frac{2\log^2(N/2)}N\right)^{1/26}
 >8p^{-1/26}.
\]
On the other hand, \eqref{eq:num-iii-mixing} is equivalent to
\[
 t^{3/2}>
 \frac{32W(N)}{\delta\sqrt p}
 \left(\frac pN\right)^2
 \sqrt{\frac{16-t}{1-t}}.
\]
Using $W(N)<3p^{1/3}$, $(p/N)^2<2$, and $t<1/2$, we find
\[
 \frac{32W(N)}{\delta\sqrt p}
 \left(\frac pN\right)^2
 \sqrt{\frac{16-t}{1-t}}
 <
 \frac{768\sqrt2}{\delta}\,p^{-1/6}.
\]
Since
\[
 p^{17/156}>X^{17/156}>534668>
 \frac{48}{\delta},
\]
the last expression is less than
$16\sqrt2\,p^{-3/52}$, which in turn is less than
$t^{3/2}$. This proves \eqref{eq:num-iii-mixing}.

Finally, elementary simplification yields
\[
 \frac{p-3}{2}-\frac{\sigma N^2}{2p}
 =
 \frac{\delta p}{2}+\sigma-\frac32-\frac{\sigma}{2p}
 >
 \frac{\delta p}{2}-1
 >
 \frac p{30000}.
\]
The function $x^{1/6}/L(x)$ is increasing for $x\geq X$, and
\[
 \frac{X^{1/6}}{L(X)}>2.25\times10^7>90000.
\]
Therefore
\[
 W(N)\sqrt p\,L(p)
 <3p^{5/6}L(p)
 <\frac p{30000},
\]
which proves \eqref{eq:num-iii-degree}.

\end{proof}

\end{document}